\documentclass[11pt]{article}

\usepackage{amssymb,amsfonts,amsmath,cite}
\usepackage{enumerate}
\usepackage{tikz}
\usepackage{forloop}

\newcommand{\pch}{\chi_{\rho}}
\newcommand{\pcS}{\chi_{S}}

\newcommand{\qed}{\hfill $\square$ \bigskip}

\newtheorem{theorem}{Theorem}[section]
\newtheorem{corollary}[theorem]{Corollary}

\newtheorem{proposition}[theorem]{Proposition}

\newtheorem{conjecture-conclude}{Conjecture}
\newtheorem{problem-conclude}[conjecture-conclude]{Problem}

\newcommand{\NN}{\mathbb N}
\newcommand{\ZZ}{\mathbb Z}

\newcommand\bonusspiral{} 
\def\bonusspiral[#1](#2)(#3:#4)(#5:#6)[#7]{
	\pgfmathsetmacro{\domain}{#4+#7*360}
	\pgfmathsetmacro{\growth}{180*(#6-#5)/(pi*(\domain-#3))}
	\draw [#1,
	shift={(#2)},
	domain=#3*pi/180:\domain*pi/180,
	variable=\t,
	smooth,
	samples=int(\domain/5)] plot ({\t r}: {#5+\growth*\t-\growth*#3*pi/180})
}

\begin{document}

\title{$S$-packing colorings of distance graphs $G(\mathbb{Z},\{2,t\})$}

\author{
Bo\v{s}tjan Bre\v{s}ar$^{a,b}$  \and Jasmina Ferme$^{c,a,b}$ \and Karol\' ina Kamenick\' a$^{d}$
}

\date{\today}

\maketitle

\begin{center}
$^a$ Faculty of Natural Sciences and Mathematics, University of Maribor, Slovenia\\
\medskip

$^b$ Institute of Mathematics, Physics and Mechanics, Ljubljana, Slovenia\\
\medskip

$^c$ Faculty of Education, University of Maribor, Slovenia\\
\medskip

$^d$ University of West Bohemia, Plze\v n, Czech Republic
\\

\end{center}

\begin{abstract}
Given a graph $G$ and a non-decreasing sequence $S=(a_1,a_2,\ldots)$ of positive integers, the mapping $f:V(G) \rightarrow \{1,\ldots,k\}$ is an $S$-packing $k$-coloring of $G$ if for any distinct vertices $u,v\in V(G)$ with $f(u)=f(v)=i$ the distance between $u$ and $v$ in $G$ is greater than $a_i$. The smallest $k$ such that $G$ has an $S$-packing $k$-coloring is the $S$-packing chromatic number, $\chi_S(G)$, of $G$. In this paper, we consider the distance graphs $G(\mathbb{Z},\{2,t\})$, where $t>1$ is an odd integer, which has $\ZZ$ as its vertex set, and $i,j\in\ZZ$ are adjacent if $|i-j|\in\{2,t\}$.
We determine the $S$-packing chromatic numbers of the graphs $G(\mathbb{Z},\{2,t\})$, where $S$ is any sequence with $a_i\in\{1,2\}$ for all $i$. In addition, we give lower and upper bounds for the $d$-distance chromatic numbers of the distance graphs $G(\mathbb{Z},\{2,t\})$, which in the cases $d\ge t-3$ give the exact values. Implications for the corresponding $S$-packing chromatic numbers of the circulant graphs are also discussed. 
\end{abstract}


\medskip\noindent
{\bf AMS Subj.\ Class: 05C15, 05C12}

\section{Introduction}
\label{sec:intro}
Let $G$ be a graph, and let $S=(a_1,a_2,\ldots)$ be a non-decreasing sequence of positive integers. 
A mapping $f:V(G)\rightarrow \{1,2,\ldots\}$ such that for every two vertices $u,v$ with $f(u)=f(v)=i$ the distance $d_G(u,v)$ is bigger than $a_i$, is {\em an $S$-packing coloring} of a graph $G$. (The {\em distance} $d_G(u,v)$ is the length of a shortest path between $u$ and $v$ in $G$.) If such a mapping $f$ exists, then $G$ is {\em $S$-packing colorable}.
 If there exists an integer $k$ such that the range $R_f$ of $f$ is $\{1,\ldots,k\}$, then $f$ is an {\em $S$-packing $k$-coloring}. The {\em $S$-packing chromatic number} of $G$, denoted by $\pcS(G)$, is the smallest integer $k$ such that $G$ admits an $S$-packing $k$-coloring.  For the sequence $S=(i)_{i\in\mathbb{N}}$ (of natural numbers in the standard order) the above concepts are known under the names {\em packing coloring}, {\em packing $k$-coloring} and the {\em packing chromatic number}, $\pch(G)$, of $G$. These concepts were introduced by Goddard, S.M.~Hedetniemi, S.T.~Hedetniemi, Harris, and Rall~\cite{goddard-2008} with a motivation coming from distibuting broadcast frequences to radio stations.  The current names were given in the second paper on the topic~\cite{bkr-2007}, and a number of authors considered these colorings later on; see a recent survey of Bre\v{s}ar, Ferme, Klav\v{z}ar and Rall~\cite{BFKR}. In addition, for an integer $d$ the sequence $S=(d,d,\ldots)$ yields the {\em $d$-distance coloring} of a graph, which has also been extensively studied; see a survey~\cite{kramer}.

A lot of attention was given to $S$-packing colorings, where $S$ contains only integers $1$ and $2$, since they lie between the classical colorings (where $S=(1,1,\ldots)$) and $2$-distance colorings (where $S=(2,2,\ldots)$). In addition, an intriguing conjecture on the packing chromatic numbers of subdivisions of cubic graphs~\cite{balogh-2019,bkrw-2017b, gt-2016} initiated several related studies. Gastineau and Togni proved that every subcubic graph is $(1, 1, 2, 2, 2)$-packing colorable~\cite{gt-2016}. Bre\v{s}ar, Klav\v{z}ar, Rall, and Wash~\cite{bkrw-2017b} investigated which subcubic graphs are $(1, 1, 2, 2)$-packing colorable, and this was continued by Liu, Liu, Rolek, and Yu~\cite{liu-2019+}.
Goddard and Xu in~\cite{goddard-2012} studied $S$-packing colorings in the two-way infinite path for different sequences $S$, which was initiated already in the seminal paper~\cite{goddard-2008}. For instance, they proved that $\chi_S(P_\infty)=3$ if and only if $S$ starts with either $(1,2,3)$, $(1,3,3)$, or $(2,2,2)$; see~\cite{ght-2019} for some additional results. Goddard and Xu~\cite{goddard-2014} considered $S$-packing colorings of the square lattice and some other related infinite graphs. Concerning the complexity issues, they proved that determining whether a graph is $(1, 1, k)$-packing colorable for any $k\ge 1$, or $(1, 2, 2)$-packing colorable, are NP-hard problems~\cite{goddard-2012}. These being very difficult problems in general, some authors were searching for classes of graphs and sequences $S$ in which the computational complexity of determining the $S$-packing chromatic number is polynomial; see such an investigation of Gastineau~\cite{g-2015}.

For a finite set of positive integers $D = \{d_1,\dots, d_k\}$, the {\em integer distance graph} $G(\mathbb{Z},D)$ with respect to $D$ is the infinite graph with $\mathbb{Z}$ as the vertex set and two distinct vertices $i, j \in \mathbb{Z}$ are adjacent if and only if $|i-j| \in D$.
The study of packing colorings of distance graphs was initiated by Togni~\cite{togni-2014}. He focused mainly on packing colorings of the distance graph $G(\mathbb{Z},D)$, where $1\in D$.  The study of the packing chromatic number of these graphs was continued by Ekstein, Holub, and Lidick\'y~\cite{ekstein-2012}, Ekstein, Holub, and Togni~\cite{ekstein-2014}, and Shao and Vesel~\cite{shao-2015}. 	
In this paper, the $S$-packing chromatic numbers of the distance graph $G(\mathbb{Z},D)$, where $D = \{2,t\}$ are considered. Note that if $t$ is even, then $G(\mathbb{Z},D)$ consists of two components, which are both isomorphic to $G(\mathbb{Z},\{1,t/2\})$. We restrict our attention only on the case $G(\mathbb{Z},\{2,t\})$, where $t>2$ is an odd integer.

To simplify the notation in this paper, we let $G(\mathbb{Z},\{2,t\})$ be denoted by $G_t$, where $t\ge 3$ is an odd integer.
Hence, $V(G_t)= \mathbb{Z}$ and $i, j \in E(G_t)$ if and only if $|i - j| = 2$ or $|i - j| = t$. 
From results in~\cite{togni-2014,ekstein-2014}, one can derive that $\pch(G_3)=13$, and from~\cite{togni-2014,shao-2015} one can infer $14\le \pch(G_5)\le 15$. In addition, Ekstein et al.~\cite{ekstein-2014} proved that for $k$ even and $t\ge 923$ and $t$ relatively prime to $k$, $\pch(G(\mathbb{Z},\{k,t\}))\le 56$; in particular, we have $\pch(G_t)\le 56$ for any odd integer $t\ge 923$. In addition, $\pch(G_t)\ge 12$ for $t\geq 9$. On the other hand, it appears that $S$-packing colorings of distance graphs, where $S\ne (1,2,3,\ldots)$, have not yet been considered, and we want to make the first step in this direction. 

In this paper, the main focus is on $(a_1,a_2,\ldots)$-packing colorings of graphs $G_t$ with $a_i\in \{1,2\}$ for all $i$, which is done in Section~\ref{sec:main}. We start the section by determining that the chromatic number of $G_t$ equals $3$, and then deal with the sequences $S$ that start with either two $1$s or one $1$ and all other integers in the sequence are $2$s. In the former case (i.e., $S=(1,1,2,2,\ldots)$), we have $\chi_S(G_t) = 4$, while in the latter case (i.e., $S=(1,2,2,\ldots)$), $\chi_S(G_t) = 5$ if $t\ge 5$ and $\chi_S(G_3) = 6$. In Section~\ref{sec:dist} we deal with distance colorings of distance graphs. We give some lower and upper bounds in the general case of $d$-distance colorings for arbitrary $d$, and in some cases when $d\ge t-3$, we determine the exact values of the $d$-distance chromatic numbers. We also prove the exact values of the $2$-distance chromatic numbers of graphs $G_t$. In the final section we give some remarks about the $S$-packing colorings of the circulant graphs, which are related to the constructions (patterns) presented in this paper.

\section{$S$-packing colorings of graphs $G_t$}
\label{sec:main}

In this section, we determine the $S$-packing chromatic numbers of graphs $G_t$ for all sequences $S$ that start with $1$. As it turns out, there are only a few cases that need to be considered for the sequence $S=(a_1,a_2,\ldots)$ where $a_1=1$. First, if $a_i=1$ for all $i$, that is, the standard chromatic number, is dealt with in Subsection~\ref{ss:111}. We prove that $\chi(G_t)=3$, which means that only the cases when there are two $1$s in $S$ or there is just one $1$ in $S$ are left to be considered. The case when $a_1=1=a_2$ and $a_i=2$ for $i\ge 3$ is considered in Subsection~\ref{ss:1122}, while the case when $a_1=1$ and $a_i=2$ for all $i$ is considered in Subsection~\ref{ss:1222}.

For a positive integer $k$, let $[k]=\{1,\ldots,k\}$ and let $[k]_0=\{0,1,\ldots,k\}$. In proofs of upper bounds, we will often say that a coloring $f:\ZZ\rightarrow [k]$ {\em uses a pattern}
$$c_1\ldots c_{\ell},$$
where $c_i\in [k]$ for all $i\in [\ell]$. By this we mean that the subsequence of colors $c_1\ldots c_\ell$ is repeatedly given to consecutive vertices of $\ZZ$; for instance and without loss of generality: $f(i+j\ell)=c_i$ for all $i\in [\ell]$ and all $j\in \ZZ$. 

The term {\em packing coloring} comes from the concept of packing. Given a graph $G$ and a positive integer $i$, a set of vertices $S\subseteq V(G)$ is an {\em $i$-packing} if for every distinct vertices $u,v\in S$, $d_G(u,v)>i$. (The  integer $i$ in an $i$-packing $S$ is sometimes called the {\em size} of the packing $S$.) In particular, a $1$-packing is an {\em independent set}. Clearly, (proper) coloring of vertices of a graph is equivalent to partitioning the vertex set to independent sets. Similarly, $S$-packing coloring is equivalent to partitioning the vertex set into packings of sizes that appear in $S$. 

\subsection{$S = (1,1,\ldots)$}
\label{ss:111}
		
In this subsection, we determine the chromatic number of $G_t$ for all $t\ge 3$; in terms of $S$-packing colorings this is the $S$-packing chromatic number for $S=(1,1,1,\ldots)$. 
		
\begin{theorem} \label{111}
For any odd integer $t\ge 3$, $\chi(G_t) = 3$.
\end{theorem}
\begin{proof}
To prove the lower bound $\chi(G_t) \ge 3$ it suffices to see that every $G_t$ has an odd cycle.  Indeed, the cycle $C:0,2,\ldots, 2t, t,0$ has $t+2$ vertices, which is an odd integer. 

Now, let us prove the upper bound for $\chi_S(G_t) \le 3$. We use the greedy (first-fit) algorithm, so that when we color a vertex $i$ its color depends on the colors of vertices $i-2$ and $i-t$. In the worst case these colors are distinct, and we still have one color available to color the vertex $i$. After coloring the first $t$ vertices, we do the same (by using the decreasing order, $0,-1,-2,\ldots,-t$) for the vertices left of the starting vertex, and then continue in the same way by alternating the directions.  \qed
\end{proof}

\subsection{$S = (1,1,2,2, \ldots)$}
\label{ss:1122}

Here we prove that $G_t$ is $(1,1,2,2)$-packing colorable. Moreover, the $S$-packing chromatic number of $G_t$ with $S=(1,1,2,2,\ldots)$ is $4$. 
	
\begin{theorem}\label{1122}
If $t\ge 3$ and $S=(1,1,2,2, \ldots)$, then $\chi_S(G_t) = 4$.
\end{theorem}
\begin{proof}
First, we prove that $\chi_S(G_t) \ge 4$ for $S=(1,1,2,2,\ldots)$, and the proof is by contradiction. Therefore, suppose that $G_t$ is $(1,1,2)$-packing colorable, and let $f:V(G)\rightarrow [3]$ be a $(1,1,2)$-packing coloring of $G_t$. Note that the vertices $v$ with $f(v)=1$, respectively $f(v)=2$, form an independent set, while the vertices $v$ with $f(v)=3$ form a $2$-packing of $G_t$. Assume without loss of generality that $f(2+t)=3$. Since $P:2,0,t,2t,2t+2$ is a path in $G_t$ all vertices of which are at distance at most $2$ from $2+t$, they are colored with a color distinct from $3$. We may assume without loss of generality that $f(2)=f(t)=f(2t+2)=1$. The path $Q:2,4,\ldots,2t$ is of even length (i.e., has an even number of edges), since $t$ is odd, and we have $f(2)=1$, $f(4)=2$. It is possible that for some $k\in \{3,\dots, t-1\}$, there is a vertex $x=2k$ such that $f(2k)=3$. Note that all vertices in the path $Q'=2k-2,2k-2+t,2k+t,2k+t+2,2k+2$ are at distance at most $2$ from $2k$, which implies that they are not given color $3$ by $f$. We infer that $f(2k-2)=f(2k+t)=f(2k+2)$. That is, vertices on the path $Q$ alternate between colors $1$ and $2$ with possible exceptions of vertices, which are colored by $3$, but in that case(s) the alternating nature of the colors on $Q$ does not change. Since $Q$ is of even length, we derive $f(2)=f(2t)=1$. This is a contradiction, since $f(2t)=1=f(t)$ and vertices $t$ and $2t$ are adjacent. 

 Next, we prove that $\chi_S(G_t) \le 4$ for $S=(1,1,2,2,\ldots)$ by forming an $S$-packing 4-coloring $f$ of the vertices of $G_t$.

\textbf{Case 1.} $t=4k-1$ \\
Let $f(j(4k+1))=3$ for all $j \in \ZZ$, $2 \nmid j$, and let $f(j(4k+1))=4$ for all $j \in \ZZ$, $2 | j$. Further, let $f(j(4k+1)+\ell)=1$ and $f(j(4k+1)+m)=2$ for any $j \in \ZZ$, $\ell \equiv 1, 2 \pmod 4$, $1 \leq \ell < 4k+1$ and $m \equiv 3, 4 \pmod 4$, $3 \leq m < 4k+1$. In this way, consecutive integers in $\mathbb{Z}$ are colored with the following pattern of colors: $$(1122)^k3(1122)^k4,$$ where $(1122)^k$ means the repetition of $1122$ $k$-times. We claim that the described coloring is an $S$-packing $4$-coloring, where $S=(1,1,2,2, \ldots)$. 

First, consider any two vertices $a, b \in V(G_t)$ such that $f(a)=f(b)=1$. This implies that there exist $j, j' \in \ZZ$ and $\ell, \ell' \in \ZZ$ with the properties that $1 \leq \ell, \ell' < (4k+1)$ (actually $\ell, \ell' \leq 4k-2$), 
$\ell \equiv 1, 2 \pmod 4$, $\ell' \equiv 1, 2 \pmod 4$ such that $a=j(4k+1)+\ell$ and $b=j'(4k+1)+\ell'$.
Suppose that $d_{G_t}(a,b)=1$ and without loss of generality assume that $a > b$. Then, $a-b \in \{2,t\}$. If $j=j'$, then $a-b=\ell-\ell'$. Clearly, $\ell-\ell' \neq 2$ and $\ell-\ell' \leq 4k-3 < t$, a contradiction to our assumption. Next, let $j \geq j'+2$. In this case, $a-b \geq 4k+5$. Therefore, $a-b \notin \{2,t\}$, again a contradiction to our assumption. Further, consider the case when $j=j'+1$. Then $a-b=4k+1+\ell-\ell'$, which is obviously greater than $2$. Now, $a-b=t$ implies that $\ell'-\ell=2$, but this is not possible since $\ell \equiv 1, 2 \pmod 4$ and $\ell' \equiv 1, 2 \pmod 4$. These findings imply that any two distinct vertices of $G_t$, both colored by $1$, are at distance at least $2$ in $G_t$. Analogously one can prove that the same holds for any two distinct vertices of $G_t$, both colored by $2$.

Next, suppose that there exist two distinct vertices $a=j(4k+1)$ and $b=j'(4k+1)$ of $G_t$, such that $j>j'$, $f(a)=f(b)=i \in \{3,4\}$ and $d_{G_t}(a,b) \leq 2$. Then, $a-b = (j-j')(4k+1) \in \{2, t, 4, 2+t, 2t, t-2\}$. If $a-b=(j-j')(4k+1) \in \{2,4\}$, then $4k+1=1$ 
and hence $k=0$, a contradiction to $t$ being a positive integer. Further, let $a-b=(j-j')(4k+1) \in \{2+t, t, t-2\}$. Since each of the integers from $\{2+t, t, t-2\}$ is odd, we derive that $j-j'$ is also odd. But this is a contradiction to the fact that $f(a)=f(b)$ which implies that both $j$ and $j'$ are divisible by $2$ or both are not divisible by $2$ and hence $j-j'$ is even. In the case when  $a-b=2t$, we derive that $(j-j'-2)4k=-j+j'-2$. Since $(j-j'-2)4k \geq 0$ and $-j+j'-2 <0$, we have a contradiction. Therefore, any two distinct vertices of $G_t$, both colored by $i \in \{3,4\}$ are at distance at least $3$ in $G_t$.

\textbf{Case 2.} $t=4k+1$ \\
Let $f(j(4k+3))=3$ for all $j \in \ZZ$, $2 \nmid j$, and $f(j(4k+3))=4$ for all $j \in \ZZ$, $2 | j$. Next, let $f(j(4k+3)+\ell)=1$ and $f(j(4k+3)+m)=2$ for any $j \in \ZZ$, $\ell \equiv 2, 3 \pmod 4$, $2 \leq \ell < 4k+3$ and $m \equiv 0, 1 \pmod 4$, $1 \leq m < 4k+3$. By the described coloring, the consecutive integers of $\ZZ$ are colored with the following pattern of colors: $$(1122)^k132(1122)^k142,$$ where $(1122)^k$ means the repetition of $1122$ $k$-times. We prove that the described coloring is an $S$-packing $4$-coloring of $G_t$, where $S=(1,1,2,2, \ldots)$. 

First, let $a, b \in V(G_t)$ such that $f(a)=f(b)=1$. This implies that there exist $j, j' \in \ZZ$ and $\ell, \ell'$ with the properties that $2 \leq \ell, \ell' < (4k+3)$, $\ell \equiv 2, 3 \pmod 4$ and $\ell' \equiv 2, 3 \pmod 4$ such that $a=j(4k+3)+\ell$ and $b=j'(4k+3)+\ell'$. We prove that  $d_{G_t}(a,b) > 1$. Suppose to the contrary that $d_{G_t}(a,b)=1$ and without loss of generality assume that $a > b$. This implies that $a-b \in \{2,t\}$. 
If $j=j'$, then $a-b=\ell-\ell'$. Clearly, $\ell-\ell' \neq 2$ and $\ell-\ell' \leq 4k < t$, a contradiction to our assumption. Next, if $j \geq j'+2$, then $a-b \geq 4k+6$ and hence $a-b \notin \{2, t\}$, a contradiction. Finally, let  $j=j'+1$. Clearly, $a-b \neq 2$, hence $a-b=t$. This implies that $\ell'-\ell=2$, but this is not possible since $\ell \equiv 2, 3 \pmod 4$ and $\ell' \equiv 2, 3 \pmod 4$. Therefore, any two distinct vertices of $G_t$, both colored by $1$, are at distance at least $2$ in $G_t$. Analogously one can prove that the same holds for any two distinct vertices of $G_t$, both colored by $2$.

Now, suppose that there exist two distinct vertices $a=j(4k+3)$ and $b=j'(4k+3)$ of $G_t$ such that $j>j'$, $f(a)=f(b)=i \in \{3,4\}$ and $d_{G_t}(a,b) \leq 2$. Then, $a-b = (j-j')(4k+3) \in \{2, t, 4, 2+t, 2t, t-2\}$. If $a-b=(j-j')(4k+3) \in \{2,4\}$, then $4k+3=1$, which is not possible since $k$ is a positive integer. Next, let $a-b=(j-j')(4k+3) \in \{2+t, t, t-2\}$. Since each of the integers from ${2+t, t, t-2}$ is odd, we infer that $j-j'$ is also odd. But, due to the fact that $f(a)=f(b)$, both $j$ and $j'$ are divisible by $2$ or both are not divisible by $2$ and hence $j-j'$ is even, so we have a contradiction. 
In the case when  $a-b=2t$, we derive that $(j-j'-2)4k=-3(j-j')+2$. Since $(j-j'-2)4k \geq 0$ and $-3(j-j')+2 <0$, we have a contradiction. These findings imply that any two distinct vertices of $G_t$, both colored by $i \in \{3,4\}$ are at distance at least $3$. This completes the proof.
\qed
\end{proof}

\subsection{$S = (1,2,2,\dots)$}
\label{ss:1222}

In this subsection, we prove that $G_t$ is $(1,2,2,2,2)$-packing colorable for every odd integer $t>3$, and that the $S$-packing chromatic number of $G_t$ with $S=(1,2,2,2,2,\ldots)$ is $5$. On the other hand, if $t=3$, then $\chi_S(G_t)=6$. 

\begin{theorem}
\label{thm:12222}
If $S = (1,2,2,2,2,2, \ldots)$, then $\chi_S(G_3) = 6$.
\end{theorem}
\begin{proof}
First, prove that $\chi_S(G_3) \geq 6$. Suppose to the contrary that $\chi_S(G_3)\leq 5$ and denote by $f$ an $S$-packing $5$-coloring of $G_3$. Suppose that there exists $i\in V(G_3)$ such that $f(i)=1=f(i+1)$. Then note that for any two integers $j,k$ from the set $A=\{i-2,i-1,i+2,i+3,i+4\}$ we have $|j-k|\in [6]$, which implies that $d_{G_3}(j,k)\le 2$. Hence, since $f$ is an $S$-packing coloring and $f(j)\ne 1\ne f(k)$, we infer $f(j)\ne f(k)$. Thus, the five vertices in $A$ should get pairwise distinct colors from $\{2,3,4,5\}$, which is a contradiction. 

From the above paragraph we derive that no two consecutive integers can get color $1$ by $f$. Therefore, at least five of the values $f(0),f(1),f(2),f(3),f(4),f(5),f(6)$ are different from $1$, but again the integers from $\{0,\ldots,6\}$ are at pairwise distance at most $2$ in $G_t$, hence those integers that are not colored by $1$ should get distinct colors. This is a contradiction, which yields $\chi_S(G_3) \geq 6$.

Next, prove that $\chi_S(G_3) \leq 6$. We form an $S$-packing $6$-coloring of $G_3$ in such a way that we color the consecutive vertices of $\ZZ$ using the following pattern of colors: $$1123411562113451162311456.$$ 
It is easy to observe  that any two distinct vertices of $G_3$ colored by $1$ are non-adjacent. Next, for any two distinct vertices $a,b \in V(G_t)$, both colored by $i \in \{2,3,4,5,6\}$, we have $|a-b| \geq 7$, which implies that they are at distance at least $3$ in $G_3$. Hence, the described coloring is an $S$-packing $6$-coloring of $G_3$, thus $\chi_S(G_3) \leq 6$. 
\qed
\end{proof}

\begin{theorem}\label{t122}
For any odd integer $t>3$, $\chi_S(G_t) = 5$ if $S = (1,2,2,2,2,\ldots)$.
\end{theorem}
\begin{proof}
Let $S = (1,2,2,2,2,\ldots)$ be a sequence. The lower bound, $\chi_S(G_t) \ge 5$, is easy to see. Let $f$ be an $S$-packing $k$-coloring of $G_t$, and let
$x\in \ZZ$ have $f(x)=1$. Since $G_t$ is $4$-regular, $x$ has four neighbors in $G_t$, which are pairwise at distance $2$. Since none of them can be colored by $1$, their colors must be pairwise distinct, that is, $f(N(x))=\{2,3,4,5\}$. Thus, $\chi_S(G_t) \ge 5$.

Next, prove that $\chi_S(G_t) \le 5$ by forming an
$S$-packing 5-coloring of $G_t$. We consider three cases.
		
{\bf Case 1.} $t=4k+1$ for $k\in \NN$. \\
In the case when $k=1$, we color the consecutive vertices of $\ZZ$ using the following pattern of colors: $$11221331144155.$$ It is clear that any two distinct vertices of $G_5$, colored by $1$, are non-adjacent. Further, let $a,b \in V(G_5)$, where $a>b$, be both colored by $i \in \{2,3,4,5\}$. Since the length of the above described sequence is $14$, we infer that $a$ and $b$ are at distance at least $3$. Therefore, the described coloring is an $S$-packing $5$-coloring of $G_t$.

Next, let $k \geq 2$. Note that a graph $G_t$ can be presented using two  infinite spirals and $t$ lines; see~\cite{ekstein-2012,ekstein-2014}. Namely, two two-way infinite spirals are drawn in parallel and then $t$ lines are added in such a way that each of them is orthogonal to both of the spirals. We denote the lines by $l_0, l_1, \ldots, l_{t-1}$ and the set of intersections between each line $l_i$, $i \in \{0, 1, \ldots, t-1\}$, and the spirals by $L_i=\{2i+jt$, $j \in \ZZ\}$. Note that $L_0=\{\ldots, -2t, -t, 0, t, 2t,\dots\}$ and $L_{t-1} = \{\ldots, -2, t-2, 2t-2, 3t-2, 4t-2, \dots\}$. See Figure~\ref{fig:v1v8}.

Now, we form the coloring $f$ of $G_t$ as follows. We color all vertices from $L_0$ one after another starting with the vertex $0$ using the following pattern of colors: $2345$. Next, for all $i \in \{2, 4, \ldots, t-1\}$, let $f(2i+jt)=1$ for all odd integers $j$, and for all $i \in \{1, 3, \ldots, t-2\}$ let $f(2i+jt)=1$ for all even integers $j$. It is clear that any two distinct vertices of $G_t$ both colored by $1$ are non-adjacent.

Further, we partition the set of still uncolored vertices of $G_t$ into $8$ subsets as follows: \\
$V_1= \{kt+2i;~ i \equiv 1\pmod 4, k \equiv 1\pmod 4\};\\
V_2 = \{kt+2i; ~i \equiv 1\pmod 4, k \equiv 3\pmod 4\};\\
V_3 = \{kt+2i; ~i \equiv 2\pmod 4, k \equiv 2\pmod 4\};\\
V_4 = \{kt+2i; ~i \equiv 2\pmod 4, k \equiv 0\pmod 4\};\\
V_5 = \{kt+2i; ~i \equiv 3\pmod 4, k \equiv 1\pmod 4\};\\
V_6 = \{kt+2i; ~i \equiv 3\pmod 4, k \equiv 3\pmod 4\};\\
V_7 = \{kt+2i; ~i \equiv 0\pmod 4, k \equiv 2\pmod 4\};\\
V_8 = \{kt+2i; ~i \equiv 0\pmod 4, k \equiv 0\pmod 4\}.$\\
Clearly, the sets $V_1, V_2, \ldots, V_8$ are disjoint and $V_1 \cup V_2 \cup \ldots \cup V_8$ presents the union of all uncolored vertices of $G_t$. Now, color all vertices from  $V_1\cup V_6$ with color $5$, all vertices from $V_2\cup V_5$ with color $3$, all vertices from $V_3\cup V_8$ with color $2$ and all vertices from $V_4\cup V_7$ with color $4$ (see Figure \ref{fig:v1v8}).

\begin{figure}
\centering
\scalebox{0.75}{
\begin{tikzpicture}
\draw (-30:9.3) node [text=black]{$\ell_{5}$}; 
\draw  (-70:9) node [text=black]{$\ell_{4}$}; 
\draw (-110:8.8) node [text=black]{$\ell_3$}; 
\draw  (-150:8.6) node [text=black]{$\ell_2$}; 
\draw (-190:8.4) node [text=black]{$\ell_1$}; 
\draw  (-230:8.2)node [text=black]{$\ell_0$}; 
\draw  (-270:9.8) node [text=black]{$\ell_{t-1}$}; 
\draw (-310:9.6) node [text=black]{$\ell_{t-2}$}; 
\node[mark size=3pt,color=black] at (-303:1.8302){\pgfuseplotmark{*}};	
\node[mark size=3pt,color=black] at (-305.6:2.8102){\pgfuseplotmark{*}};	
\node[mark size=3pt,color=black] at (-307:3.802){\pgfuseplotmark{*}};	
\node[mark size=3pt,color=black] at (-308.7:4.802){\pgfuseplotmark{*}};	
\node[mark size=3pt,color=black] at (-309.2:5.802){\pgfuseplotmark{*}};	
\node[mark size=3pt,color=black] at (-309.6:6.802){\pgfuseplotmark{*}};	
\node[mark size=3pt,color=black] at (-309.7:7.802){\pgfuseplotmark{*}};	
\node[mark size=3pt,color=black] at (-309.9:8.802){\pgfuseplotmark{*}};	
\draw (-310:2.02) node [text=black]{$1$}; 
\draw  (-310.6:3.02) node [text=black]{$5$}; 
\draw (-310.3:4.02) node [text=black]{$1$}; 
\draw (-310.8:5.02) node [text=black]{$3$}; 
\draw (-311.1:6.02) node [text=black]{$1$}; 
\draw (-311.1:7.02) node [text=black]{$5$}; 
\draw (-311.2:8.02) node [text=black]{$1$}; 
\draw (-311.4:9.02) node [text=black]{$3$}; 	
\node[mark size=3pt,color=black] at (-263.5:2.03){\pgfuseplotmark{*}};	 
\node[mark size=3pt,color=black] at (-266.5:3.03){\pgfuseplotmark{*}};	 
\node[mark size=3pt,color=black] at (-267.5:4.03){\pgfuseplotmark{*}};	 
\node[mark size=3pt,color=black] at (-268.5:5){\pgfuseplotmark{*}};	 
\node[mark size=3pt,color=black] at (-269:6){\pgfuseplotmark{*}};	 
\node[mark size=3pt,color=black] at (-269.5:7){\pgfuseplotmark{*}};	 
\node[mark size=3pt,color=black] at (-269.8:8){\pgfuseplotmark{*}};	 
\node[mark size=3pt,color=black] at (-270:9){\pgfuseplotmark{*}};	
\draw (-268.5:2.3) node [text=black]{$4$}; 
\draw  (-269.5:3.3) node [text=black]{$1$}; 
\draw (-270.5:4.3) node [text=black]{$2$}; 
\draw (-271:5.3) node [text=black]{$1$}; 
\draw (-271:6.3) node [text=black]{$4$}; 
\draw (-271:7.3) node [text=black]{$1$}; 
\draw (-271:8.3) node [text=black]{$2$}; 
\draw (-271:9.3) node [text=black]{$1$}; 

\draw (-255.5:2.3) node [text=black]{$_{(-2)}$}; 
\draw  (-256.9:3.3) node [text=black]{$_{(-2+t)}$}; 
\draw (-259.3:4.3) node [text=black]{$_{(-2+2t)}$}; 
\draw (-262.3:5.3) node [text=black]{$_{(-2+3t)}$}; 
\draw (-263.6:6.3) node [text=black]{$_{(-2+4t)}$}; 
\draw (-264.65:7.3) node [text=black]{$_{(-2+5t)}$}; 
\draw (-265.6:8.3) node [text=black]{$_{(-2+6t)}$}; 
\draw (-266.3:9.27) node [text=black]{$_{(-2+7t)}$}; 
\node[mark size=3pt,color=black] at (-226.2:2.2383){\pgfuseplotmark{*}};	 			 
\node[mark size=3pt,color=black] at (-227.4:3.2383){\pgfuseplotmark{*}};	 	 
\node[mark size=3pt,color=black] at (-228.6:4.2383){\pgfuseplotmark{*}};	 
\node[mark size=3pt,color=black] at (-229.2:5.2383){\pgfuseplotmark{*}};	
\node[mark size=3pt,color=black] at (-229.4:6.2383){\pgfuseplotmark{*}};	
\node[mark size=3pt,color=black] at (-229.9:7.2383){\pgfuseplotmark{*}};		
\draw (-231.9:2.483) node [text=black]{$2$}; 
\draw  (-231.5:3.483) node [text=black]{$3$}; 
\draw (-231.5:4.483) node [text=black]{$4$}; 
\draw (-231.3:5.483) node [text=black]{$5$}; 
\draw (-231.4:6.483) node [text=black]{$2$}; 
\draw (-231.7:7.473) node [text=black]{$3$}; 

\draw (-221.7:2.493) node [text=black]{$_{(0)}$}; 
\draw  (-224.9:3.493) node [text=black]{$_{(t)}$}; 
\draw (-225.3:4.498) node [text=black]{$_{(2t)}$}; 
\draw (-226.6:5.493) node [text=black]{$_{(3t)}$}; 
\draw (-227.1:6.493) node [text=black]{$_{(4t)}$}; 
\draw (-227.9:7.493) node [text=black]{$_{(5t)}$}; 
\node[mark size=3pt,color=black] at (-186.2:2.4599){\pgfuseplotmark{*}};		  		  		 
\node[mark size=3pt,color=black] at (-187.4:3.4599){\pgfuseplotmark{*}};	
\node[mark size=3pt,color=black] at (-188.5:4.4599){\pgfuseplotmark{*}};	
\node[mark size=3pt,color=black] at (-189:5.4599){\pgfuseplotmark{*}};	
\node[mark size=3pt,color=black] at (-189.5:6.4529){\pgfuseplotmark{*}};	
\node[mark size=3pt,color=black] at (-189.9:7.4529){\pgfuseplotmark{*}};		
\draw  (-191:2.699) node [text=black]{$1$}; 
\draw  (-191.1:3.639) node [text=black]{$5$}; 
\draw (-191.1:4.699) node [text=black]{$1$}; 
\draw (-191.5:5.699) node [text=black]{$3$}; 
\draw (-191.7:6.69) node [text=black]{$1$}; 
\draw (-191.8:7.69) node [text=black]{$5$}; 	

\draw  (-181:2.699) node [text=black]{$_{(2)}$}; 
\draw  (-183.8:3.839) node [text=black]{$_{(t+2)}$}; 
\draw (-185.1:4.899) node [text=black]{$_{(2t+2)}$}; 
\draw (-186.5:5.910) node [text=black]{$_{(3t+2)}$}; 
\draw (-187.7:6.92) node [text=black]{$_{(4t+2)}$}; 
\draw (-187.8:7.89) node [text=black]{$_{(5t+2)}$}; 	
\node[mark size=3pt,color=black] at (-146.2:2.6992){\pgfuseplotmark{*}};	 
\node[mark size=3pt,color=black] at (-147.4:3.6792){\pgfuseplotmark{*}};	
\node[mark size=3pt,color=black] at (-148.5:4.6792){\pgfuseplotmark{*}};	
\node[mark size=3pt,color=black] at (-149:5.6732){\pgfuseplotmark{*}};	
\node[mark size=3pt,color=black] at (-149.5:6.6702){\pgfuseplotmark{*}};	
\node[mark size=3pt,color=black] at (-149.9:7.6702){\pgfuseplotmark{*}};		
\draw  (-151:2.8992) node [text=black]{$4$}; 
\draw  (-150.9:3.8992) node [text=black]{$1$}; 
\draw (-151:4.892) node [text=black]{$2$}; 
\draw (-151.1:5.892) node [text=black]{$1$}; 
\draw (-151.3:6.892) node [text=black]{$4$}; 
\draw (-151.5:7.892) node [text=black]{$1$}; 	
\node[mark size=3pt,color=black] at (-106.2:2.9142){\pgfuseplotmark{*}};	
\node[mark size=3pt,color=black] at (-107.4:3.9042){\pgfuseplotmark{*}};	
\node[mark size=3pt,color=black] at (-108.5:4.9042){\pgfuseplotmark{*}};	
\node[mark size=3pt,color=black] at (-109:5.9042){\pgfuseplotmark{*}};	
\node[mark size=3pt,color=black] at (-109.5:6.9042){\pgfuseplotmark{*}};	
\node[mark size=3pt,color=black] at (-109.9:7.9042){\pgfuseplotmark{*}};	
\draw  (-101.2:2.7142) node [text=black]{$1$}; 
\draw  (-103.8:3.7042) node [text=black]{$3$}; 
\draw (-105.6:4.65042) node [text=black]{$1$}; 
\draw (-106.5:5.65042) node [text=black]{$5$}; 
\draw (-107.6:6.65042)node [text=black]{$1$}; 
\draw (-108.15:7.65042) node [text=black]{$3$}; 
\node[mark size=3pt,color=black] at (-58.8:1.1802){\pgfuseplotmark{*}};		  		  		 
\node[mark size=3pt,color=black] at (-64.3:2.1402){\pgfuseplotmark{*}};	
\node[mark size=3pt,color=black] at (-66.8:3.1102){\pgfuseplotmark{*}};	
\node[mark size=3pt,color=black] at (-68.25:4.1102){\pgfuseplotmark{*}};	
\node[mark size=3pt,color=black] at (-68.75:5.102){\pgfuseplotmark{*}};	
\node[mark size=3pt,color=black] at (-69.4:6.102){\pgfuseplotmark{*}};	
\node[mark size=3pt,color=black] at (-69.6:7.102){\pgfuseplotmark{*}};	
\node[mark size=3pt,color=black] at (-69.9:8.102){\pgfuseplotmark{*}};	
\draw  (-53:1.4402) node [text=black]{$4$}; 
\draw  (-61:2.4402) node [text=black]{$1$}; 
\draw (-64:3.4102) node [text=black]{$2$}; 
\draw (-65.7:4.41022) node [text=black]{$1$}; 
\draw (-66.8:5.402)node [text=black]{$4$}; 
\draw (-67.2:6.402) node [text=black]{$1$}; 
\draw (-68.2:7.402)node [text=black]{$2$}; 
\draw (-68.3:8.402) node [text=black]{$1$}; 
\node[mark size=3pt,color=black] at (-22.4:1.3802){\pgfuseplotmark{*}};	
\node[mark size=3pt,color=black] at (-26:2.3502){\pgfuseplotmark{*}};	
\node[mark size=3pt,color=black] at (-27.5:3.3502){\pgfuseplotmark{*}};	
\node[mark size=3pt,color=black] at (-28.3:4.3502){\pgfuseplotmark{*}};	
\node[mark size=3pt,color=black] at (-28.9:5.3502){\pgfuseplotmark{*}};	
\node[mark size=3pt,color=black] at (-29.4:6.3502){\pgfuseplotmark{*}};	
\node[mark size=3pt,color=black] at (-29.6:7.3502){\pgfuseplotmark{*}};	
\node[mark size=3pt,color=black] at (-29.9:8.3502){\pgfuseplotmark{*}};	
\draw  (-17:1.6802) node [text=black]{$1$}; 
\draw  (-22:2.6502) node [text=black]{$3$}; 
\draw (-25:3.650) node [text=black]{$1$}; 
\draw (-26:4.6502) node [text=black]{$5$}; 
\draw (-27:5.6502)node [text=black]{$1$}; 
\draw (-27.5:6.6502) node [text=black]{$3$}; 
\draw (-28:7.6502)node [text=black]{$1$}; 
\draw (-28.3:8.6502) node [text=black]{$5$}; 
\bonusspiral [dashed] [black](0,0)(90:270)(-1:-8)[3];
\bonusspiral [dashed] [black](0,0)(90:270)(-2:-9)[3];
\bonusspiral[black](0,0)(90:130)(-1:-1.4422)[0.1111];
\bonusspiral[black](0,0)(90:130)(-2:-2.4422)[0.1111];
\bonusspiral[black](0,0)(-150:130)(-1.67:-3.4422)[0.1111];
\bonusspiral[black](0,0)(-150:130)(-2.67:-4.4422)[0.1111];
\bonusspiral[black](0,0)(-150:130)(-3.67:-5.4422)[0.1111];
\bonusspiral[black](0,0)(-150:130)(-4.67:-6.4422)[0.1111];
\bonusspiral[black](0,0)(-150:130)(-5.67:-7.4422)[0.1111];
\bonusspiral[black](0,0)(-150:130)(-6.67:-8.4422)[0.1111];
\bonusspiral[black](0,0)(210:240)(-7.67:-8.05)[0.1111];
\bonusspiral[black](0,0)(210:240)(-8.67:-9.05)[0.1111];
\draw (-20:1) -- (-30:8.8);
\draw (-50:0.7) -- (-70:8.6);
\draw (-105:2.4) -- (-110:8.4);
\draw (-145:2.2) -- (-150:8.2);
\draw (-185:2.0) -- (-190:8);
\draw (-225:1.8) -- (-230:7.8);
\draw (-262:1.6) -- (-270:9.4);
\draw (-300:1.4) -- (-310:9.2);
\end{tikzpicture}
}
\caption[]{Representation of $G_t$ with two disjoint infinite spirals and $t$ lines. Vertices of $\ZZ$ are in small size surrounded by brackets, while colors are shown in normal size and are as in Case 1 of the proof of Theorem \ref{t122}.}
\label{fig:v1v8}
\end{figure}
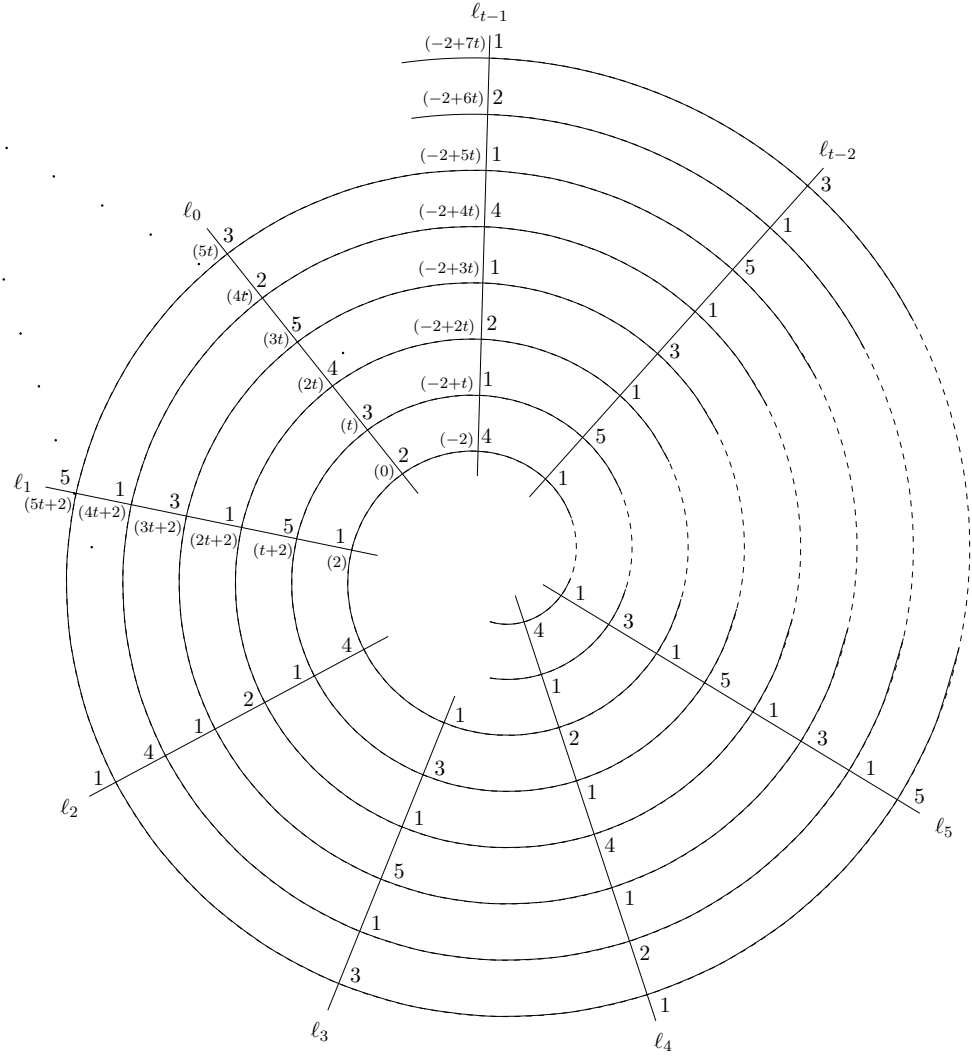

It remains to prove that any two vertices $a, b \in V(G_t)$, colored with the same color $s \in \{2,3,4,5\}$, are at distance at least $3$. Clearly, if $a, b \in L_i$ for some $i \in \{0, 1, \ldots, t-1\}$, then they are at distance at least $4$. Hence, suppose that $a \in L_m$ and $b \in L_n$, where $0 \leq m < n \leq t-1$. If $(n-m)>2 \pmod t$, then $d_{G_t}(a,b) \geq 3$. 
If $n-m =1\pmod t$, then we have two cases: $a\in L_0, b\in L_1$ and $a \in L_0,b\in L_{t-1}$. In either case we can write $a=j_mt$ and $b=\pm 2+j_nt$, $j_m, j_n \in \ZZ$. Since $f(a)=f(b)$, we have $|j_m-j_n| \geq 2$ and hence $d_{G_t}(a,b) \geq 3$. 
Finally, suppose that $n-m=2 \pmod t$. Let $a=2m+j_mt$ and $b=2n+j_nt$, $j_m, j_n \in \ZZ$. Since $f(a)=f(b)$, we have that $|j_m - j_n| \geq 2$ and thus $d_{G_t}(a,b) \geq 4$. Therefore, $f$ is a $S$-packing $5$-coloring of $G_t$.
%

{\bf Case 2.} $t=4k-1$ for $k\in \NN$, and $3\nmid t$. \\
In this case, color the consecutive vertices of $G_t$ one after another using the following pattern of colors:
$$123145.$$
Note that for any two vertices $a, b \in V(G_t)$, both colored by $1$, we have $a-b=3m$, $m \in \ZZ$. This implies that $a$ and $b$ are not adjacent in $G_t$. Next, if any two vertices $c, d \in V(G_t)$ are both colored by $s \in \{2,3,4,5\}$, then $c-d =6i$, $i \in \ZZ$. Clearly, $c-d \notin \{2, 4, t\}$, and $c-d \notin \{t+2, t-2\}$, since $t+2$ and $t-2$ are odd integers, but $c-d$ is even. Moreover, $c-d=6i \neq 2t$, since $3 \nmid t$. Therefore, the distance between $c$ and $d$ in $G_t$ is greater than $2$, which implies that the described coloring is an $S$-packing $5$-coloring of $G_t$ and hence $\chi_S(G_t) \leq 5$.

\medskip

{\bf Case 3.} $t=4k-1$ for $k\in \NN$, $k \geq 2$, and $3 | t$. \\
Let $k \geq 2$ be an arbitrary positive integer and define the coloring $f$ of the vertices of $G_t$ as folows. 
Let $f(j(4k-3))=1$ and $f(j(4k-3)+\ell)=1$ for every $j \in \ZZ$, $\ell \equiv 1 \pmod 3$, $1 \leq \ell \leq 4k-4$ (actually,  $1 \leq \ell \leq 4k-6$). 
Further,  for every $m \in \ZZ$, $m \equiv 2,3 \pmod 6$, $1 \leq m \leq 4k-4$ (since $4k-4$ is even and divisible by $3$, we have $2 \leq m \leq 4k-7$), let  $f(j(4k-3)+m)=2$ if $2|j$, and otherwise $f(j(4k-3)+m)=4$. Finally, for every $p \in \ZZ$, $p \equiv 0,5 \pmod 6$, $1 \leq p \leq 4k-4$ ($5 \leq p \leq 4k-4$), let  $f(j(4k-3)+p)=3$ if $2|j$, and otherwise $f(j(4k-3)+p)=5$. 
In this way, the consecutive vertices of $\ZZ$ are colored with the following pattern of colors: $$1(122133)^{v}1(144155)^v,$$ where $v= \frac{2(k-1)}{3}$. 

Now, we prove that $f$ is an $S$-packing $5$-coloring of $G_t$. 
First, consider a vertex $a \in V(G_t)$, colored by $1$.  If $a=j(4k-3)$ for some $j \in \ZZ$, then its neighbours are $j(4k-3)+2$, $j(4k-3)-2$, $j(4k-3)+t$ and $j(4k-3)-t$. Clearly, $f(j(4k-3)+2) \neq 1$ and $f(j(4k-3)-2) \neq 1$. Using the fact that each sequence $(122133)^v$ (respectively, $(144155)^v$)  contains $t-3$ integers, we derive that $f(j(4k-3)+t) \neq 1$ and $f(j(4k-3)-t)) \neq 1$. Next, suppose that $a=j(4k-3)+\ell$ for some $j \in \ZZ$, $\ell \equiv 1 \pmod 3$, $1 \leq \ell \leq 4k-4$. Again, it is clear that $a+2$ and $a-2$ do not receive color $1$ by $f$. Since $a+t=(j+1)(4k-3)+(\ell+2)$ and $\ell+2 \equiv 0 \pmod 3$, we derive that $f(a+t) \neq 1$. Analogously, $a-t=(j-1)(4k-3)+(\ell-2)$ and $l-2 \equiv 2 \pmod 3$, hence $f(a-t) \neq 1$. These findings imply that any two vertices of $G_t$, both colored by $1$, are at distance at least $2$.

Next, let $b \in V(G_t)$ such that $f(b) \in \{2,3,4,5\}$. Clearly, $f(b-2) \neq f(b)$, $f(b+2) \neq f(b)$, $f(b-4) \neq f(b)$ and $f(b+4) \neq f(b)$. Next, since each sequence $(122133)^v$ (respectively, $(144155)^v$)  contains $t-3$ integers, we have $f(b \pm (t-2)) \neq f(b)$ and $f(b \pm t) \neq f(b)$. Further, since $f(b) \in \{2,3,4,5\}$, $b=j(4k-3)+x$ for some positive integers $j, x$. Hence, $b+2t=(j+2)(4k-3)+x+4$, which implies that $f(b+2t) \neq f(b)$. Analogously, $b-2t=(j-2)(4k-3)+x-4$, $b+(t+2)=(j+1)(4k-3)+x+4$ and $b-(t+2)=(j-1)(4k-3)+x-4$. Thus, $f(b - 2t) \neq f(b)$ and $f(b \pm (2+t)) \neq f(b)$, which implies that $f$ is an $S$-packing $5$-coloring of $G_t$ and the proof is done.
 \qed
\end{proof}
		
\section{Distance coloring of graphs $G_t$}
\label{sec:dist}
		A distance coloring relative to distance $d$ of a graph $G$ is a mapping $V(G)\rightarrow \{1,2,3,\dots\}$ such that any two distinct vertices $a,b\in V(G)$ with $f(a)=f(b)$ are at distance greater than $d$ in $G$.
Note that for the sequence $S = (d,d,d,\dots)$, an $S$-packing coloring presents the distance coloring relative to distance $d$, that is, the $d$-distance coloring.

\subsection{Lower bound for distance colorings of $G_t$}
In this section we present a lower bound for $\chi_S(G_t)$ with $S=(d,d,d,d,\dots)$. 
		
\begin{theorem} \label{lowdist}
If $t\ge 3$ is an odd integer, $S = (d,d,d,\dots)$, and $d\ge\frac{t+1}{2}$
then $$\chi_S(G_t) \ge 1+t\left(d-\frac{t-3}{2}\right).$$ 
\end{theorem}
\begin{proof}
We claim that every $1+t\left(d-\frac{t-3}{2}\right)$ consecutive integers in $\ZZ$ are at pairwise distance at most $d$ in $G_t$. It is clear that this claim implies the bound of the theorem. Without loss of generality consider the subsequence of consecutive integers starting with $0$. To prove this claim it suffices to show that all integers in the set $V=[t\left(d-\frac{t-3}{2}\right)]=\{1,\ldots, t(d-\frac{t-3}{2})\}$ are at distance at most $d$ from $0$ in $G_t$. Let $n = d-\frac{t-3}{2}$, so that we can write $V=[nt]$. 

Suppose $y\in [nt]$ is an odd integer. Let $i\in \NN_0$ such that $t(2i-1)< y\le t(2i+1)$. Note that $i\in \{0,\ldots, \lfloor\frac {n}{2}\rfloor\}$. If $y<2it+2$, then there exists $r\in [\frac{t+1}{2}]_0$ such that $y=(2i-1)t+2r$, and  
$$P:0,t,\ldots,(2i-2)t,(2i-1)t,(2i-1)t+2,\ldots,(2i-1)t+2r-2,y$$ 
is a path between $0$ and $y$ whose length is $2i-1+r$. Now, 
$$2i-1+r\le n-1+\frac{t+1}{2}=d-\frac{t-3}{2}-1+\frac{t+1}{2}=d+1,$$
hence $d_{G_t}(0,y)\le d$ unless $2i-1=n-1$ and $r=\frac{t+1}{2}$. However, if $y=(n-1)t+2\frac{t+1}{2}=nt+1$, then $y$ is not in $[nt]$. 

On the other hand, if $y> 2it+2$, then there exists $r\in [\frac{t-3}{2}]_0$ such that $y=(2i+1)t-2r$, and 
$$P:0,t,\ldots,2it,(2i+1)t,(2i+1)t-2,\ldots,(2i+1)t-(2r-2),y$$
is a path between $0$ and $y$ whose length is $2i+1+r$. Since $y\in[nt]$, we have $2i+1\le n$, and so   
$$2i+1+r\le n+\frac{t-3}{2}=d-\frac{t-3}{2}+\frac{t-3}{2}=d,$$
hence $d_{G_t}(0,y)\le d$.
Note that if $y=2it+2$, then $y$ is an even integer.

If $y\in [nt]$ is an even integer, by a similar analysis as above one can show  $d_{G_t}(0,y)\le d$. As noted in the beginning of the proof, we derive that every $1+t\left(d-\frac{t-3}{2}\right)$ consecutive integers in $\ZZ$ are at distance at most $d$, which implies that they all get distinct colors. Therefore, $\chi_S(G_t) \ge 1+t\left(d-\frac{t-3}{2}\right).$
\qed
\end{proof}


\subsection{Upper bound for distance colorings of $G_t$}
 In this section we present an upper bound for $\chi_S(G_t)$ with $S=(d,d,d,\dots)$.
\begin{theorem}\label{updist}
If $t\ge 3$ is an odd integer and $S = (d,d,d,\dots)$, then
 		
$$\chi_S(G_t) \le \left\{ \begin{array}{lll}
 			\ 1+d(d+1)& ; & d\le \frac{t+1}{2}\\
 			\ {td+\frac{1}{4}(-t^2+2t+7)}&; &  d\ge \frac{t+1}{2} .
\end{array} \right.$$
 	\end{theorem}
\begin{proof}
We start with $d\le \frac{t+1}{2}$. We use the greedy (first-fit) algorithm so that when a vertex $i$ is colored we use the smallest possible color that was not given to already colored vertex at distance at most $d$ from a given vertex. After coloring the first $t$ vertices (from $0$ to $t-1$), we then do the same (by using the decreasing order $i-1,i-2,\ldots$) for $t$ vertices left of $0$, and then continue in the same way by alternating the directions.  

The color of a vertex, say $0$, thus depends on the colors of vertices from the following sets: 
\begin{itemize}
	\item[$V_1=$] $\{-2,-4,-6,\dots, -2d\}$,
	\item [$V_2=$] $\{-t,-2t,-3t,\dots, -td\}$,
	\item [$V_3=$] $\{ -kt \pm 2\ell\,|\, k\in [d-1]$, $\ell\in [d-1]$ and $k+\ell \le d\}$.
\end{itemize}
Clearly, all integers in $V_1\cup V_2\cup V_3$ are left of $0$, and $|V_1| =|V_2| = d$ and $|V_3| = \sum\limits_{k=1}^{d-1}2(d-k) = d(d-1)$. Hence the color of $0$ depends on the colors of $d(d+1)$ vertices to the left of $0$. In the worst case these colors are distinct, and we still have one color available to color the vertex $0$. This holds for an arbitrary vertex, hence the bound $\chi_S(G_t) \le 1+d(d+1)$ follows. 

Now we prove $\chi_d(G_t) \le \ td+\frac{1}{4}(-t^2+2t+7)$ for $d\ge \frac{t+1}{2}$.
	Again we use the greedy (first-fit) algorithm by starting at an arbitrary vertex (say, $0\in \ZZ$) and color the vertices in the increasing order by using the smallest possible color that was not given to already colored vertex at distance at most $d$ from $x$. After coloring the first $t$ vertices, we then do the same by using the decreasing order for the vertices left of $0$, and then continue in the same way by alternating the directions. 
It is clear that the color of a vertex $0$ may depend only on the colors of vertices in $\{-1,-2,\ldots,-td\}$. Not all of these vertices are at distance at most $d$ from $0$, and we next determine which are not. 
	
Let	$V = \{-t(d-1)-2(x+1)\,|\,x\in\{1,\ldots, \frac{t+1}{2}-2\}\}$. It is easy to see that for a vertex $v\in V$, we have $d_{G_t}(0,v)>d$.	
	%
Note that $|V| = \frac{t+1}{2}-2$. 

Next, let $W=\{-t(d-y)+2(y+z)\,|\, y\in\{0,1,\ldots,\frac{t-3}{2}-1\}, z\in \{1,\ldots, t-3-2y\}\}$. We claim that $d_{G_t}(0,w)>d$ for any $w\in W$. 
Consider the following two types of paths: 
 $$R: 0,-t,\dots,-t(d-y),-t(d-y)+2,\dots, -t(d-y)+2y=r$$ and $$S = 0,-t,\dots,-t(d-y-2),-t(d-y-2)-2,\dots, -t(d-y-2)-2(y+2)=s.$$ 
	
Note that every (distinct) path from $R\cup S$ with length at most $d$ starting at vertex $0$ does not contain a vertex $u$ such that $r<u<s$ except a path $T$, where $$T: 0,-t,\dots,-t(d-y-1),-t(d-y-1)\pm2,\dots, -t(d-y-1)\pm 2(j+1),j\le y.$$ But then the vertex $t\in T$, where $r<t<s$, are odd and vertices of $W$ are even, or vice versa (depending on $y$). Hence a path $T$ does not contain any vertex of $W$.

Now we prove that $r<w$ and $w<s$ for every $w\in W$. Since $r= -t(d-y)+2y$  and $w = -t(d-y)+2y+2z$ reduces to $0< 2z$, the first inequality is correct. In the proof of the second inequality, from $w = -t(d-y)+2y+2z\le  -t(d-y)+2y+2(t-3-2y)$, we get $w\le (-td+ty+2t-2y-4)-2<-td+ty+2t-2y-4=s.$
Since $r<w<s$ for all $w\in W$, there does not exist a path of length at most $d$ between vertices $0$ and $w$ in $G_t$.  Hence the color of $0$ does not depend on the colors of vertices in $W$. Note that $|W| = \sum\limits_{y=0}^{\frac{t-3}{2}-1}2(\frac{t-3}{2}-y)=  \frac{t-3}{2}(\frac{t-3}{2}+1)$. \\
	
We infer that the color of the vertex $0$ depends on the colors of at most  $dt-|V|-|W|$ vertices, and $dt-|V|-|W| = dt - (\frac{t+1}{2}-2) - \frac{t-3}{2}(\frac{t-3}{2}+1) = td + \frac{1}{4}(-t^2+2t+3)$.
	In the worst case these colors are distinct, and we still have one color available to color the vertex $0$, and $td + \frac{1}{4}(-t^2+2t+3)+1 = td + \frac{1}{4}(-t^2+2t+7)$, which equals the announced bound.
\qed
\end{proof}


\subsection{Some exact values of the $d$-distance chromatic numbers of $G_t$}
In this subsection, we prove the exact values of $\chi_d(G_t)$ for $d\ge t-3$. 

If $t=3$, then Theorems~\ref{lowdist} and~\ref{updist} yield the exact value of $\chi_d(G_3)$. 
\begin{corollary}\label{col}
	If $d\ge 2$ is an integer, then $\chi_d(G_3) = 3d+1$. 
\end{corollary}

\begin{theorem}\label{exactd}
If $t\ge 5$ is an odd integer  and $d\ge t-3$, then $$\chi_d(G_t) = 1+t\left(d-\frac{t-3}{2}\right).$$
\end{theorem}
\begin{proof}
Let $\ell=1+t\left(d-\frac{t-3}{2}\right)$. From Theorem~\ref{lowdist} it follows that $\chi_d(G_t) \geq \ell$, so it remains to prove that $\chi_d(G_t) \leq \ell$. Let $f$ be a coloring of the vertices of $G_t$ obtained by using the following pattern on the consecutive vertices of $\ZZ$: 
$$123\dots(\ell-1)\ell.$$ 

We claim that $f$ is a $d$-distance coloring for $G_t$. First, observe that $td < 2 \ell$. This means that $d_{G_t}(a,b) \leq d$ implies $|b-a| < 2 \ell$ for any two integers $a$ and $b$. 
Therefore, for any $i\in \ZZ$, we only need to check the distances between the vertices $i, i+1 \ldots, i+2\ell-1$ in $G_t$. Moreover, from the definitions of $f$ and the $d$-distance coloring, if suffices to prove that $d_{G_t}(0,\ell) \geq d+1$. Suppose to the contrary that $d_{G_t}(0,\ell)=d' \leq d$. 


First, if $d'\le d-\frac{t-3}{2}$, then $d't < \ell$, which means that $d_{G_t}(0,\ell) > d'$, a contradiction.
Hence, let $d'= d-\frac{t-3}{2}+x$, where $x \in \{1, 2, \ldots, \frac{t-3}{2}\}$. We may write, $\ell - 0 = \ell =pt\pm2r$, where $p$ and $r$ are positive integers and $p+r=d'$. We distinguish three cases. 

\textbf{Case 1.} $p \geq d' - x +1$. \\
Since $p+r=d'$, we have $r \leq x-1$. 
	From the fact that $pt > \ell$ we derive that only the following vertices lie on a shortest $0,\ell$-path:
	$$p't-2r',p'\in [d'-x+1]_0, r'\in[ d'-p]_0.$$
	Note that only a vertex $p't-2r'$, where $p'\geq d'-x+1$ and $r'\in [d'-p]$, is greater than $\ell$. However, since $r \leq \frac{t-3}{2}-1$, we have $pt-2r > \ell$ whenever $p\geq d'-x+1$. This implies that each of the shortest $0,\ell$-paths does not contain a vertex $\ell$, a contradiction. 
	
\textbf{Case 2.} $p = d' - x$. \\
In this case we derive that $pt = \ell-1$ and hence in each of the shortest $0,\ell$-paths $P$ only the following vertices lie: $$p't+2r',p'\in [d'-x]_0, r'\in[d'-p]_0.$$
	Note that only a vertex $p't+2r'$, where $p'= d'-x$  and $r'>1$ is greater than $\ell$, but such a vertex has the same parity as $pt$. Since $pt = \ell-1$ , vertex $\ell$ has distinct parity as $pt$, and so $\ell$ does not lie on $P$, a contradiction. 

\textbf{Case 3.} $p \leq d' - x -1$. \\
From $pt < \ell$ we infer that on each of the shortest $0,\ell$-paths $P$ only the following vertices lie: $$p't+2r',p' \in [d' - x -1]_0, r'\in[d'-p]_0.$$
Note that the biggest vertex of these paths is $pt+2r$, where $p = d' - x -1$ and $r =  d'-p$ 
and with some calculation we prove that $pt+2r\le \ell-2$ and hence $pt+2r< \ell$.
Since $r = d'-p$ we have
$$pt+2r = pt+2(d'-p) ,$$ 
and using $p= d'-x-1$, we get
$$pt+2r= ( d'-x-1)t+2d'-2( d'-x-1).$$
Further we use $d'= d-\frac{t-3}{2}+x$, hence we get 
$$pt+2r=\left( d-\frac{t-3}{2}\right)t-t+2x+2.$$
Since $x\le \frac{t-3}{2}$, we derive
$$pt+2r\le  \left( d-\frac{t-3}{2}\right) t-1,$$ and from $ \left( d-\frac{t-3}{2}\right)t-1 = \ell -2$ we infer the desired inequality 
$$pt+2r\le \ell-2<\ell,$$ implying that $P$ does not contain $\ell$, a contradiction. \\

Therefore, there does not exist a $0,\ell$-path of length $d'$ for any $d' \leq d$, and so $f$ is indeed an $S$-packing $\ell$-coloring of $G_t$. The proof is complete. \qed
\end{proof}

\subsection{The $2$-distance chromatic number of $G_t$}\label{2dist}

In this subsection, we consider the $S$-packing coloring of $G_t$, where $S=(2,2,2,\ldots)$. For a set of integers $\{a_1,\ldots, a_r\}$ where $a_i< m$ for all $i$, we write $a\equiv a_1,\ldots, a_r \pmod m$ if $a\equiv a_i \pmod m$ for some $i\in [r]$.
	
\begin{theorem}\label{2dist}
If $t>3$ is an odd integer, then
$$\chi_2(G_t)=\left\{
        \begin{array}{ccl}
                5 &;&  t \equiv 1, 9 \pmod {10} \\
                6 &;&  t \equiv 3,5,7 \pmod {10} \\
        \end{array}
       \right.,$$
and $\chi_2(G_3)=7$. 
\end{theorem}
\begin{proof}
The lower bound $\chi_2(G_t)\ge 5$ is trivial. Indeed, since $G_t$ is $4$-regular, there are five vertices in the closed neighborhood $N[v]$ of a vertex $v\in V(G_t)$, and they must receive pairwise distinct colors. 

First, consider the case $t \equiv 1, 9 \pmod {10}$. Let $f:\mathbb{Z} \rightarrow [5]$ be a coloring produced by the following pattern of colors:
$$12345.$$
Clearly, for any two vertices $u,v\in \mathbb{Z}$ with $f(u)=f(v)$ we get $d_{\mathbb{Z}}(u,v)=5k$ for some $k\in \mathbb{N}$. On the other hand, note that two distinct vertices $u,v\in\mathbb{Z}$ are at distance at most $2$ in $G_t$ only if $|u-v|\equiv 1,2,3,4,7,8,9 \pmod {10}$, since either $t \equiv 1 \pmod {10}$ or $t \equiv 9 \pmod {10}$.  

Second, consider the case $t \equiv 3,5,7 \pmod {10}$ and $t>3$. For the upper bound $6$, we deal with small cases separately, that is, when $t\in\{5,7,13\}$ take a coloring $f$ produced by the following pattern of colors:
$$123456.$$
Clearly, for any two vertices $u,v\in \mathbb{Z}$ with $f(u)=f(v)$ we get $|u-v|=6k$ for some $k\in \mathbb{N}$. On the other hand, for $t=5$, two vertices $u,v\in\mathbb{Z}$ are at distance at most $2$ in $G_t$ only if $|u-v|\in\{2,3,4,5,7,10\}$; for $t=7$, two vertices $u,v\in\mathbb{Z}$ are at distance at most $2$ in $G_t$ only if $|u-v|\in\{2,4,5,7,9,14\}$; and, for $t=13$, two vertices $u,v\in\mathbb{Z}$ are at distance at most $2$ in $G_t$ only if $|u-v|\in\{2,4,11,13,15,26\}$. Since none of the integers in the distance sets is divisible by $6$, the coloring $f$ is a $(2,2,2,2,2,2)$-packing coloring of $G_t$ for all $t\in\{5,7,13\}$. 

Now, let $t\ge 15$, and let $a$ and $\ell<5$ be the unique integers such that $t+1=5a+\ell$. Next, let $k=a-\ell$. Consider the following pattern of colors given to consecutive integers:
$$(12345)^k(123456)^\ell.$$
Note that the basic unit of the pattern $(12345)^k(123456)^{\ell}$ consist of exactly $t+1$ colors, hence an integer $x$ colored by a color $i\in [6]$ is at $\mathbb{Z}$-distance $t$ to a vertex whose color precedes or follows the place of the color of $x$ in the basic unit of the pattern; that is, $i-1$ or $i+1$ (with respect to either modulo $5$ or modulo $6$ depending in which place of the basic unit of the pattern the color is taken). Similarly, an integer $x$ is at $\mathbb{Z}$-distance $2t$ to a vertex whose color is two after or two before the place of the color of $x$ in the basic unit of the pattern, and the same holds, of course, also for the distance $2$. It is also easy to see that an integer $x$ is at $\mathbb{Z}$-distance $t\pm 2$ to an integer which is either three places before, or one place before, or one place after, or three places after the place of the color of $x$ in the basic unit of the pattern. Combining these observations with the implication $$d_{G_t}(x,y)\le 2\implies d_{\mathbb{Z}}(x,y)\in \{2,t\pm 2,2t\},$$
we infer that the color of $x$ is different from the colors of integers at distance at most $2$ from $x$ in $G_t$.

Now we prove the lower bound of $\chi_2(G_t)\ge6$ for $t\equiv 3,5,7 \pmod{10}$ and $t>3$. The proof is by contradiction, thus suppose that $G_t$ is $(2,2,2,2,2)$-packing colorable. Let $H_t$ be the graph obtained from $P_4\square P_t$ by adding the edge $(2,1)(4,t)$, where $V(H_t) =\{(i,j):i\in [4], j\in[t]$ and edges are defined in the natural way; see Figure~\ref{lowbound2}. Note that $H_t$ is a subgraph of $G_t$ (see Figure~\ref{subHt} in which $H_t$ is embedded in $G_t$). Hence, $H_t$ is $(2,2,2,2,2)$-packing colorable too. Let $f:V(H_t) \rightarrow [5]$ be a $(2,2,2,2,2)$-packing coloring. Without loss of generality, assume  $f(2,2) = 2$, $f(1,2)= 4$,  $f(3,2)= 5$, $f(2,1)= 1$ and $f(2,3)= 3$, as shown in Figure~\ref{lowbound2}. 
\begin{figure}
	\centering
	\scalebox{0.60}{
		\begin{tikzpicture}
		\node[mark size=3pt,color=black] at (-303:1.8302){\pgfuseplotmark{*}};	
		\node[mark size=3pt,color=black] at (-305.6:2.8102){\pgfuseplotmark{*}};	
		\node[mark size=4pt,color=black] at (-307:3.802){\pgfuseplotmark{*}};	
		\node[mark size=4pt,color=black] at (-308.7:4.802){\pgfuseplotmark{*}};	
		\node[mark size=4pt,color=black] at (-309.2:5.802){\pgfuseplotmark{*}};	
		\node[mark size=4pt,color=black] at (-309.6:6.802){\pgfuseplotmark{*}};	
		\node[mark size=3pt,color=black] at (-309.7:7.802){\pgfuseplotmark{*}};	
		\node[mark size=3pt,color=black] at (-309.9:8.802){\pgfuseplotmark{*}};	
		\draw[line width=1mm,color = black] (-307:3.802) -- (-309.6:6.802);
		\node[mark size=3pt,color=black] at (-263.5:2.03){\pgfuseplotmark{*}};	 
		\node[mark size=3pt,color=black] at (-266.5:3.03){\pgfuseplotmark{*}};	 
		\node[mark size=4pt,color=black] at (-267.5:4.03){\pgfuseplotmark{*}};	 
		\node[mark size=4pt,color=black] at (-268.5:5){\pgfuseplotmark{*}};	 
		\node[mark size=4pt,color=black] at (-269:6){\pgfuseplotmark{*}};	 
		\node[mark size=4pt,color=black] at (-269.5:7){\pgfuseplotmark{*}};	 
		\node[mark size=3pt,color=black] at (-269.8:8){\pgfuseplotmark{*}};	 
		\node[mark size=3pt,color=black] at (-270:9){\pgfuseplotmark{*}};	
		\draw[line width=1mm,color = black] (-267.5:4.03) -- (-269.5:7);
		\node[mark size=3pt,color=black] at (-226.2:2.2383){\pgfuseplotmark{*}};	 			 
		\node[mark size=3pt,color=black] at (-227.4:3.2383){\pgfuseplotmark{*}};	 	 
		\node[mark size=4pt,color=black] at (-228.6:4.2383){\pgfuseplotmark{*}};	 
		\node[mark size=4pt,color=black] at (-229.2:5.2383){\pgfuseplotmark{*}};	
		\node[mark size=4pt,color=black] at (-229.4:6.2383){\pgfuseplotmark{*}};	
		\node[mark size=4pt,color=black] at (-229.9:7.2383){\pgfuseplotmark{*}};		
		\draw[line width=1mm,color = black] (-228.6:4.2383) -- (-229.9:7.2383);
		\node[mark size=3pt,color=black] at (-186.2:2.4599){\pgfuseplotmark{*}};		  		  		 
		\node[mark size=3pt,color=black] at (-187.4:3.4599){\pgfuseplotmark{*}};	
		\node[mark size=4pt,color=black] at (-188.5:4.4599){\pgfuseplotmark{*}};	
		\node[mark size=4pt,color=black] at (-189:5.4599){\pgfuseplotmark{*}};	
		\node[mark size=4pt,color=black] at (-189.5:6.4529){\pgfuseplotmark{*}};	
		\node[mark size=4pt,color=black] at (-189.9:7.4529){\pgfuseplotmark{*}};		
		\draw[line width=1mm,color = black] (-188.5:4.4599) -- (-189.9:7.4529);
		\draw (-191.5:5.699) node [text=black]{${\bf 3}$}; 
		\draw (-186:5.8299) node [text=black]{${\bf _{(2,3)}}$}; 
		\node[mark size=3pt,color=black] at (-146.2:2.6992){\pgfuseplotmark{*}};	 
		\node[mark size=3pt,color=black] at (-147.4:3.6792){\pgfuseplotmark{*}};	
		\node[mark size=4pt,color=black] at (-148.5:4.6792){\pgfuseplotmark{*}};	
		\node[mark size=4pt,color=black] at (-149:5.6732){\pgfuseplotmark{*}};	
		\node[mark size=4pt,color=black] at (-149.5:6.6702){\pgfuseplotmark{*}};	
		\node[mark size=4pt,color=black] at (-149.9:7.6702){\pgfuseplotmark{*}};		
		\draw[line width=1mm,color = black] (-148.5:4.6792) -- (-149.9:7.6702);
		\draw (-151.3:4.90092) node [text=black]{${\bf 4}$}; 
		\draw (-151.4:5.90092) node [text=black]{${\bf 2}$}; 
		\draw (-151.5:6.90092) node [text=black]{${\bf 5}$}; 
		\draw  (-143.3:4.33992) node [text=black]{${\bf _{(1,2)}}$}; 
		\draw (-145:5.31292) node [text=black]{${\bf _{(2,2)}}$}; 
		\draw (-146.1:6.31292) node [text=black]{${\bf _{(3,2)}}$}; 	
		\node[mark size=3pt,color=black] at (-106.2:2.9142){\pgfuseplotmark{*}};	
		\node[mark size=3pt,color=black] at (-107.4:3.9042){\pgfuseplotmark{*}};	
		\node[mark size=4pt,color=black] at (-108.5:4.9042){\pgfuseplotmark{*}};	
		\node[mark size=4pt,color=black] at (-109:5.9042){\pgfuseplotmark{*}};	
		\node[mark size=4pt,color=black] at (-109.5:6.9042){\pgfuseplotmark{*}};	
		\node[mark size=4pt,color=black] at (-109.9:7.9042){\pgfuseplotmark{*}};	
		\draw[line width=1mm,color = black] (-108.5:4.9042) -- (-109.9:7.9042);
		\draw (-111.8:5.63042) node [text=black]{${\bf 1}$}; 
		\draw (-102.6:4.63042) node [text=black]{${\bf _{(1,1)}}$}; 
		\draw (-103.9:5.6042) node [text=black]{${\bf _{(2,1)}}$}; 
		\draw (-105.4:6.6042) node [text=black]{${\bf _{(3,1)}}$}; 
		\node[mark size=3pt,color=black] at (-58.8:1.1802){\pgfuseplotmark{*}};		  		  		 
		\node[mark size=3pt,color=black] at (-64.3:2.1402){\pgfuseplotmark{*}};	
		\node[mark size=4pt,color=black] at (-66.8:3.1102){\pgfuseplotmark{*}};	
		\node[mark size=4pt,color=black] at (-68.25:4.1102){\pgfuseplotmark{*}};	
		\node[mark size=4pt,color=black] at (-68.75:5.102){\pgfuseplotmark{*}};	
		\node[mark size=4pt,color=black] at (-69.4:6.102){\pgfuseplotmark{*}};	
		\node[mark size=3pt,color=black] at (-69.6:7.102){\pgfuseplotmark{*}};	
		\node[mark size=3pt,color=black] at (-69.9:8.102){\pgfuseplotmark{*}};	
		\draw[line width=1mm,color = black] (-66.8:3.1102) -- (-69.4:6.102);
		\draw (-73:5.772) node [text=black]{${\bf _{(4,t)}}$}; 
		\node[mark size=3pt,color=black] at (-22.4:1.3802){\pgfuseplotmark{*}};	
		\node[mark size=3pt,color=black] at (-26:2.3502){\pgfuseplotmark{*}};	
		\node[mark size=4pt,color=black] at (-27.5:3.3502){\pgfuseplotmark{*}};	
		\node[mark size=4pt,color=black] at (-28.3:4.3502){\pgfuseplotmark{*}};	
		\node[mark size=4pt,color=black] at (-28.9:5.3502){\pgfuseplotmark{*}};	
		\node[mark size=4pt,color=black] at (-29.4:6.3502){\pgfuseplotmark{*}};	
		\node[mark size=3pt,color=black] at (-29.6:7.3502){\pgfuseplotmark{*}};	
		\node[mark size=3pt,color=black] at (-29.9:8.3502){\pgfuseplotmark{*}};	
			\draw[line width=1mm,color = black] (-27.5:3.3502) -- (-29.4:6.3502);

		\bonusspiral[black](0,0)(90:130)(-1:-1.4422)[0.1111];
		\bonusspiral[black](0,0)(90:130)(-2:-2.4422)[0.1111];
		\bonusspiral[black](0,0)(-150:130)(-1.67:-3.4422)[0.1111];
		\bonusspiral[black](0,0)(-150:130)(-2.67:-4.4422)[0.1111];
		\bonusspiral[black](0,0)(-150:130)(-3.67:-5.4422)[0.1111];
		\bonusspiral[black](0,0)(-150:130)(-4.67:-6.4422)[0.1111];
		\bonusspiral[black](0,0)(-150:130)(-5.67:-7.4422)[0.1111];
		\bonusspiral[black](0,0)(-150:130)(-6.67:-8.4422)[0.1111];
		\bonusspiral[black](0,0)(210:240)(-7.67:-8.05)[0.1111];
		\bonusspiral[black](0,0)(210:240)(-8.67:-9.05)[0.1111];
		\bonusspiral[line width=1mm,color = black](0,0)(113.2:130)(-3.1102:-3.4422)[0.1111];
		\bonusspiral[line width=1mm,color = black](0,0)(111.75:130)(-4.1102:-4.4422)[0.1111];
		\bonusspiral[line width=1mm,color = black](0,0)(111.25:130)(-5.102:-5.4422)[0.1111];
		\bonusspiral[line width=1mm,color = black](0,0)(110.6:130)(-6.102:-6.4422)[0.1111];

		\bonusspiral[line width=1mm,color = black](0,0)(-150:31.5)(-3.67:-4.9042)[0.1111];
		\bonusspiral[line width=1mm,color = black](0,0)(-150:71)(-4.67:-6.102)[0.1111];
		\bonusspiral[line width=1mm,color = black](0,0)(-150:30.5)(-5.67:-6.9042)[0.1111];
		\bonusspiral[line width=1mm,color = black](0,0)(-150:30.1)(-6.67:-7.9042)[0.1111];
		
	\bonusspiral[dashed, black](0,0)(170:172)(-1.4422:-1.67)[0.1111];
	\bonusspiral[dashed, black](0,0)(170:172)(-2.4422:-2.67)[0.1111];
	\bonusspiral[dashed, line width=1mm,color = black](0,0)(170:172)(-3.4422:-3.67)[0.1111];
	\bonusspiral[dashed, line width=1mm,color = black](0,0)(170:172)(-4.4422:-4.67)[0.1111];
	\bonusspiral[dashed, line width=1mm,color = black](0,0)(170:172)(-5.4422:-5.67)[0.1111];
	\bonusspiral[dashed, line width=1mm,color = black](0,0)(170:172)(-6.4422:-6.67)[0.1111];
	\bonusspiral[dashed, black](0,0)(170:171.5)(-7.4422:-7.67)[0.1111];
	\bonusspiral[dashed, black](0,0)(170:171.5)(-8.4422:-8.67)[0.1111];
			
		\draw (-20:1) -- (-30:8.8);
		\draw (-50:0.7) -- (-70:8.6);
		\draw (-105:2.4) -- (-110:8.4);
		\draw (-145:2.2) -- (-150:8.2);
		\draw (-185:2.0) -- (-190:8);
		\draw (-225:1.8) -- (-230:7.8);
		\draw (-262:1.6) -- (-270:9.4);
		\draw (-300:1.4) -- (-310:9.2);
		\end{tikzpicture}
	}
	\caption[]{Subgraph $H_t$ (with thick edges) of the graph $G_t$.}
	\label{subHt}
\end{figure}
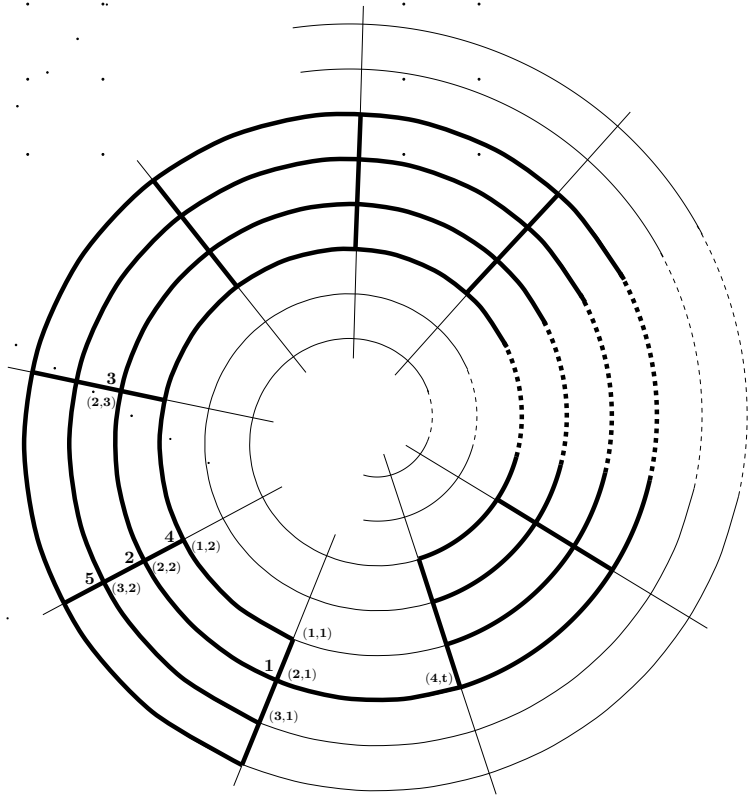

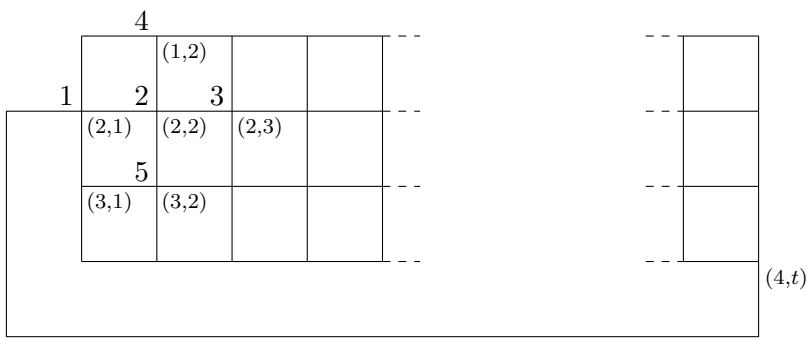
\begin{figure}
\centering
\begin{tikzpicture}
\draw (0,0) --(4,0) (0,-1) --(4,-1) (0,-2) --(4,-2) (0,-3) --(4,-3) (8,0) --(9,0) (8,-1) --(9,-1) (8,-2) --(9,-2)(8,-3) --(9,-3);
\draw (0,0)--(0,-3) (1,0)--(1,-3) (2,0)--(2,-3)(3,0)--(3,-3) (4,0)--(4,-3) (8,0)--(8,-3)(9,0)--(9,-3);
\draw (0,-1)--(-1,-1)    (-1,-1) --(-1,-4)  (-1,-4)--(9,-4)  	(9,-4)  --(9,-3);
\draw[dashed] (4,0) --(4.5,0)(4,-1) --(4.5,-1) (4,-2) --(4.5,-2) (4,-3) --(4.5,-3) (7.5,0) --(8,0) (7.5,-1) --(8,-1)(7.5,-2) --(8,-2)(7.5,-3) --(8,-3);
\node[mark size=3pt,color=black] at (0,0){\pgfuseplotmark{*}};
\node[mark size=3pt,color=black] at (0,-1){\pgfuseplotmark{*}};
\node[mark size=3pt,color=black] at (0,-2){\pgfuseplotmark{*}};
\node[mark size=3pt,color=black] at (0,-3){\pgfuseplotmark{*}};
\node[mark size=3pt,color=black] at (1,0){\pgfuseplotmark{*}};
\node[mark size=3pt,color=black] at (1,-1){\pgfuseplotmark{*}};
\node[mark size=3pt,color=black] at (1,-2){\pgfuseplotmark{*}};
\node[mark size=3pt,color=black] at (1,-3){\pgfuseplotmark{*}};
\node[mark size=3pt,color=black] at (4,-1){\pgfuseplotmark{*}};
\node[mark size=3pt,color=black] at (2,0){\pgfuseplotmark{*}};
\node[mark size=3pt,color=black] at (2,-1){\pgfuseplotmark{*}};
\node[mark size=3pt,color=black] at (2,-2){\pgfuseplotmark{*}};
\node[mark size=3pt,color=black] at (2,-3){\pgfuseplotmark{*}};
\node[mark size=3pt,color=black] at (4,-2){\pgfuseplotmark{*}};
\node[mark size=3pt,color=black] at (3,0){\pgfuseplotmark{*}};
\node[mark size=3pt,color=black] at (3,-1){\pgfuseplotmark{*}};
\node[mark size=3pt,color=black] at (3,-2){\pgfuseplotmark{*}};
\node[mark size=3pt,color=black] at (3,-3){\pgfuseplotmark{*}};
\node[mark size=3pt,color=black] at (4,-3){\pgfuseplotmark{*}};
\node[mark size=3pt,color=black] at (4,0){\pgfuseplotmark{*}};
\node[mark size=3pt,color=black] at (8,0){\pgfuseplotmark{*}};
\node[mark size=3pt,color=black] at (8,-1){\pgfuseplotmark{*}};
\node[mark size=3pt,color=black] at (8,-2){\pgfuseplotmark{*}};
\node[mark size=3pt,color=black] at (8,-3){\pgfuseplotmark{*}};
\node[mark size=3pt,color=black] at (9,0){\pgfuseplotmark{*}};
\node[mark size=3pt,color=black] at (9,-1){\pgfuseplotmark{*}};
\node[mark size=3pt,color=black] at (9,-2){\pgfuseplotmark{*}};
\node[mark size=3pt,color=black] at (9,-3){\pgfuseplotmark{*}};
\draw (0.38,-1.22) node [text=black]{$_{(2,1)}$}; 
\draw (1.38,-1.22) node [text=black]{$_{(2,2)}$}; 
\draw (2.38,-1.22) node [text=black]{$_{(2,3)}$}; 
\draw (1.38,-0.22) node [text=black]{$_{(1,2)}$}; 
\draw (1.38,-2.22) node [text=black]{$_{(3,2)}$}; 
\draw (9.38,-3.22) node [text=black]{$_{(4,t)}$}; 
\draw (0.38,-2.22) node [text=black]{$_{(3,1)}$}; 
\draw (0.8,-0.8) node [text=black]{$2$}; 
\draw (0.8,-1.8) node [text=black]{$5$}; 
\draw (0.8,0.2) node [text=black]{$4$}; 
\draw (-0.2,-0.8) node [text=black]{$1$}; 
\draw (1.8,-0.8) node [text=black]{3}; 
\end{tikzpicture}
\caption[]{Subgraph $H_t$  of the graph $G_t$.}
\label{lowbound2}
\end{figure}

We consider the following three cases for the colors of vertices $(1,1)$ and $(3,1)$ given by $f$:
\begin{itemize}
\item [1)] $f(3,1) = 4, f(1,1)=5$,
\item [2)] $f(3,1) = 4, f(1,1)=3$,
\item [3)] $f(3,1) = 3$.
\end{itemize}
Note that after coloring the vertices of $A=\{(1,1),(1,2),(2,1),(2,2),(2,3),(3,1),(3,2)\}$, colors of vertices in $V(P_4\square P_t)-(A\cup\{(4,t)\})$ are uniquely determined, since $f$ is assumed to be a $(2,2,2,2,2)$-packing coloring.

{\bf Case 1.} $f((3,1)) = 4, f((1,1))=5$.\\
Note that the colors of vertices in $A$ imply $f(1,3)=1$. Since $f(3,1)=4$, we infer that $(3,3)$ is at distance at most $2$ in $H_t$ from any color in $[5]$, which is a contradiction.

{\bf Case 2.} $f((3,1)) = 4, f((1,1))=3$.\\
As mentioned above, the colors of vertices in $V(P_4\square P_t)-(A\cup\{(4,t)\})$ are uniquely determined by the colors given to vertices of $A$ by $f$. More precisely, $f$ partitions the vertices of $V(P_4\square P_t)$ into the sets $V_\ell$, where $\ell\in [5]$, and $f(i,j)=\ell$ if and only $(i,j)\in V_{\ell}$, as follows:

$V_1=\{(1,j\equiv4\pmod 5), (2,j \equiv 1\pmod 5), (3,j \equiv 3\pmod5), (4,j \equiv 0\pmod 5)\},$

$V_2=\{(1,j \equiv 0\pmod 5), (2,j \equiv 2\pmod 5), (3,j \equiv 4\pmod5), (4,j \equiv 1\pmod 5)\},$

$V_3=\{(1,j \equiv 1\pmod 5), (2,j \equiv 3\pmod 5), (3,j \equiv 0\pmod5), (4,j \equiv 2\pmod 5)\},$

$V_4=\{(1,j \equiv 2\pmod 5), (2,j \equiv 4\pmod 5), (3,j \equiv 1\pmod5), (4,j \equiv 3\pmod 5)\},$

$V_5=\{(1,j \equiv 3\pmod 5), (2,j \equiv 0\pmod 5), (3,j \equiv 2\pmod5), (4,j \equiv 4\pmod 5)\}.$

Clearly, the sets $V_1,\ldots, V_5$ are pairwise disjoint and $V_1 \cup \cdots \cup V_5=V(H_t)$. Note that in $H_t$ the vertex $(4,t)$ is at distance at most $2$ from all vertices in $\{(1,1),(2,1),(2,2),(3,1)\}$, and these vertices are colored by colors $1,2,3$ and $4$. Hence the only color possibly available to color $(4,t)$ is $5$. According to the $2$-distance coloring $f$ obtained in Case 2, the vertex  $(4,t)$ is colored by $5$ if and only if $t\equiv 4\pmod 5$. This implies that $t\equiv 4,9\pmod {10}$, which contradicts $t\equiv 3,5,7 \pmod{10}$. Hence, $H_t$ is not $(2,2,2,2,2)$-colorable, and so $G_t$ is not $(2,2,2,2,2)$-colorable.

{\bf Case 3.} $f(3,1) = 3$.\\
This case is very similar to the case $2$, and we omit the details.

Finally, note that the pattern $$1234567$$
for coloring the elements of $\mathbb{Z}$ works in $G_3$. Indeed, $d_{G_3}(x,y)\le 2$ if and only if $d_{\mathbb{Z}}(x,y)\in\{1,2,3,4,5,6\}$, and by the above pattern, two integers $x$ and $y$ receive the same color only if $|x-y|$ is divisible by $7$. Hence, $\chi_2(G_3)\le 7$. The converse, that we need at least $7$ colors for a $2$-distance coloring of $G_3$ uses a similar argument. Since for two vertices $x$ and $y$ with the same color, $d_{\mathbb{Z}}(x,y)\notin\{1,2,3,4,5,6\}$, we infer that every seven consecutive integers must receive pairwise distinct colors in any $2$-distance coloring of $G_t$, yielding $\chi_2(G_3)\ge 7$. 
\qed
\end{proof}


\section{Concluding remarks}

Some of the results of this paper can be used for determining the $S$-packing chromatic numbers of certain circulant graphs. Recall that a {\em circulant graph} with parameters $n,s_1\ldots,s_k$ is the graph $G=C_n(s_1,\ldots, s_k)$ with $V(G)=\{0,\ldots,n-1\}$ and vertex $i\in V(G)$ is adjacent to $i+s_j \pmod n$ for all $j\in [k]$. Note that the adjacencies are defined in the same way as in distance graphs, yet the circulant graphs are finite. See~\cite{ap-1979} for an introductory study of circulant graphs and a recent paper~\cite{lp-2017}, where the so-called perfect colorings of circulant graphs are studied (in fact, the circulant graphs in the latter paper are infinite and correspond to the distance graphs).
It is well-known that 
$C_n(s_1,\ldots, s_k)$ is connected if and only if $GCD(n,s_1,\ldots, s_k)=1$. 

The patterns used in some upper bounds for the $S$-packing chromatic numbers of graphs $G_t=G(\ZZ,\{2,t\})$ in this paper can be directly applied for $S$-packing colorings of the circulant graphs $C_n(2,t)$, whenever $n$ is divisible by the length of the pattern. On the other hand, lower bounds (obtained in this paper) for the $S$-packing chromatic number of a given distance graph directly imply the same lower bound for the $S$-packing coloring of the corresponding circulant graph. We restrict our attention to circulant graphs $C_n(2,t)$ and may assume without loss of generality that $n\ge 2t$.

For the sequence $S = (1,1,1)$, Theorem~\ref{111} provides the lower bound $3$.  However, in the proof of the theorem we did not use a pattern, but a greedy algorithm. We now present the patterns that can be used for coloring of $C_n(2,t)$ when $n$ is divisible by $t+2$. 

If $t=4k+1$ (resp. $t=4k-1$), then we consider the following pattern of colors given to consecutive vertices $C_n(2,t)$:
 $$(1122)^k132 \mbox{ (resp. }(1122)^k3).$$ 
It is easy to see that this yields a proper coloring of $C_n(2,t)$, when $n$ is divisible by the length of the pattern, which is $t+2$ in both cases. 
 
\begin{proposition}
If $t\ge 3$ is an odd integer and $n=(t+2)m$, where $m\in \NN-\{1\}$, then $\chi(C_n(2,t))=3$.
\end{proposition}
 
Let $S = (1,1,2,2)$. Theorem~\ref{1122} and its proof providing the patterns for $S$-packing colorings yield the following result.
 
\begin{corollary}
 If $t\ge 3$, $S=(1,1,2,2)$ and $n=2(t+2)m$, where $m\in \NN$, then $\chi_S(C_n(2,t))=4$.
\end{corollary}

Indeed, the patterns used when $t=4k+1$ (resp. $t=4k-1$) were:
$$(1122)^k132(1122)^k142 \mbox{ (resp. }(1122)^k3(1122)^k4),$$ 
the length of each of which is $2(t+2)$. 

In Theorem~\ref{thm:12222} we proved that $\chi_S(G_3) = 6$ for $S = (1,2,2,2,2,2, \ldots)$, where for the coloring we used the pattern $1123411562113451162311456$. The pattern has length $25$, and we derive the following result.
\begin{corollary}
	If $S=(1,2,2,2,2,2)$ and $n=25m$, where $m\in \NN$, then $\chi_S(C_n(2,3))=6$.
\end{corollary}
 
When $t> 3$, the following result is a consequence of Theorem \ref{t122} and the patterns used in its proof.
  
 \begin{corollary}
 For any odd integer $t>3$ and $S=(1,2,2,2,2,2)$, then 
 $\chi_S(C_n(2,t))=5$ if 

$ \begin{array}{l@{\quad}l}
\ 	\mbox{(i)  $t= 5$, $n=14m$, where $m\in \NN$ or,}\\
\ 	\mbox{(ii) $t= 4k-1, 3\nmid t$ and $n=6m>2t$, where $m\in \NN$ or,}\\
\ 	\mbox{(iii) $t= 4k-1, 3\mid t$ and $n=2(4k-3)m$, where $m\in \NN-\{1\}$.  }\\
\end{array} $\\
\end{corollary}
 
The $S$-packing $5$-coloring when $t=4k+1$, where $k>1$, is more complex and cannot be directly applied to color $C_n(2,t)$.

Now we list some consequences of our results for distance colorings of $C_n(2,t)$.
Corollary~\ref{col} implies the following result (the pattern is simply $12\ldots (3d+1)$).
\begin{corollary}\label{col2}
If $d\ge 2$  is an integer, and $n=(3d+1)m$, where $m\in \NN$, then $\chi_d(C_n(2,3)) = 3d+1$. 
\end{corollary}

For $t\ge 5$, Theorem~\ref{exactd} implies the following result.
\begin{corollary}
If $t\ge 5$ is an odd integer, $d\ge t-3$ and $n = m(1+t\left(d-\frac{t-3}{2}\right))$, where $m \in \NN$, then 
$	\chi_d(C_n(2,t)) = 1+t\left(d-\frac{t-3}{2}\right)$.
\end{corollary}

For the upper bound we use the pattern for consecutive vertices of $C_n(2,t)$, which contains integers from $1$ to $1+t\left(d-\frac{t-3}{2}\right)$ in the natural order (each integer appears once).

Combining Theorem~\ref{2dist} and its proof with Corollary~\ref{col2} we get:

\begin{corollary}
	If $t>3$ is an odd integer, then 
	$$\chi_2(C_n(2,t))=\left\{
	\begin{array}{ccl}
	5 &;&  t \equiv 1, 9 \pmod {10} \mbox{ and } n = 5m\geq 4t,m\in\NN\\
	6 &;&  t \in\{5,7,13\} \mbox{ and } n = 6m\geq 4t,m\in\NN \\
	6 &;&  t \equiv 3,5,7 \pmod {10} , t\geq 15 \mbox{ and } n = sm\geq 4t,m\in\NN,\\
	\end{array}
	\right.$$
	where $s = 5k+6\ell, \ell<5, t+1=5a+\ell, k = a-\ell$.
	Furthermore, $\chi_2(C_n(2,3))=7$ for $n = 7m$, where $m\in\NN$. 
\end{corollary}

We conclude the paper by proposing an investigation of $S$-packing colorings of circulant graphs. The case $C_n(2,t)$ could be a starting point, where one should resolve the missing cases, which are not covered in this section. The next natural candidates to consider are the circulant graphs $C_n(1,t)$.

\section*{Acknowledgements}
We thank P\v remysl Holub for initial discussions on the topic.
B.B. and J.F. acknowledge the financial support of the Slovenian Research Agency (research core funding No.\ P1-0297 and projects J1-9109 and J1-1693).
K.K. was partially supported by the project GA20–09525S of the Czech Science Foundation.

\end{document}